\documentclass[12pt]{amsart}
\usepackage{amsmath}
\usepackage{amssymb}
\usepackage{comment}
\newtheorem{theorem}{Theorem}[section]
\newtheorem{corollary}[theorem]{Corollary}
\newtheorem{lemma}[theorem]{Lemma}
\newtheorem{prop}[theorem]{Proposition}
\theoremstyle{definition}
\newtheorem{defn}[theorem]{Definition}
\theoremstyle{remark}
\newtheorem{remark}[theorem]{\bf{Remark}}
\theoremstyle{remark}
\newtheorem{example}[theorem]{\bf{Example}}
\numberwithin{equation}{section}

\newcommand{\beas} {\begin{eqnarray*}}
\newcommand{\eeas} {\end{eqnarray*}}
\newcommand{\bes} {\begin{equation*}}
\newcommand{\ees} {\end{equation*}}
\newcommand{\be} {\begin{equation}}
\newcommand{\ee} {\end{equation}}
\newcommand{\bea} {\begin{eqnarray}}
\newcommand{\eea} {\end{eqnarray}}

\newcommand{\R}{\mathbb R}
\newcommand{\C}{\mathbb C}

\begin{document}

\title[$\sigma$-Point and nontangential convergence] {A note on $\sigma$-point and nontangential convergence}

\author[J. Sarkar]{Jayanta Sarkar}

\address{Stat-Math Unit, Indian Statistical Institute, 203, B.T. Road, Kolkata-700108, India}
\email{jayantasarkarmath@gmail.com}



\begin{abstract}
In this article, we generalize a theorem of Victor L. Shapiro concerning nontangential convergence of the Poisson integral of a $L^p$-function. We introduce the notion of $\sigma$-points of a locally finite measure and consider a wide class of convolution kernels. We show that convolution integrals of a measure have nontangential limits at $\sigma$-points of the measure. We also investigate the relationship between $\sigma$-point and the notion of the strong derivative introduced by Ramey and Ullrich. In one dimension, these two notions are the same.
\end{abstract}


\maketitle
\section{Introduction}
In this article, by a measure $\mu$ we will always mean a complex Borel measure or a signed Borel measure such that the total variation $|\mu|$ is locally finite, that is, $|\mu|(K)$ is finite for all compact sets $K$. If $\mu(E)$ is nonnegative for all Borel measurable sets $E$ then $\mu$ will be called a positive measure. The notion of Lebesgue point of a measure was defined by Saeki \cite{Sa} which we recall. A point $x_0\in\R^n$ is called a Lebesgue point of a measure $\mu$ on $\R^n$ if there exists $L\in\C$ such that
\begin{equation}\label{leblim}
\lim_{r\to 0}\frac{|\mu-Lm|(B(x_0,r))}{m(B(0,r))}=0,
\end{equation}
where $B(x_0,r)$ denotes the open ball of radius $r$ with center at $x_0$ with respect to the Euclidean metric and $m$ denotes the Lebesgue measure of $\R^n$. In this case, the symmetric derivative of $\mu$ at $x_0$,
\begin{equation}\label{symder}
D_{sym}\mu(x_0):=\lim_{r\to 0}\frac{\mu(B(x_0,r))}{m(B(0,r))}
\end{equation}
exists and is equal to $L$. The set of all of Lebesgue points of a measure $\mu$ is called the Lebesgue set of $\mu$. It is not very hard to see that the Lebesgue set of a measure $\mu$ includes almost all (with respect to the Lebesgue measure) points of $\R^n$ (see Proposition \ref{almostevery}). Given a measure $\mu$ on $\R^n$, its Poisson integral $P\mu$ on the upper half space $\R^{n+1}_+=\{(x,t):x\in\R^n,t>0\}$ is defined by the convolution
\begin{equation*}
P\mu (x,t)=\int_{\R^n}P(x-\xi,t)\:d\mu (\xi),
\end{equation*}
whenever the integral exists. Here, the kernel $P(x,t)$ is the usual Poisson kernel of $\R_+^{n+1}$ given by the formula
\begin{equation*}
P(x,t)=c_n\frac{t}{(t^2+\|x\|^2)^{\frac{n+1}{2}}},\:\:\:\:\: c_n=\pi^{-(n+1)/2}\Gamma\left(\frac{n+1}{2}\right).
\end{equation*}
It is known that if the integral above exists for some $(x_0,t_0)\in \R_+^{n+1}$ then it exists for all points in $\R_+^{n+1}$ and defines a harmonic function in $\R_+^{n+1}$. In this article, we will be concerned with the nontangential convergence of $P\mu$ or of more general convolution integrals. For $x_0\in\R^n$ and $\alpha>0$, we define the conical region $S(x_0,\alpha)$ with vertex at $x_0$ and aperture $\alpha$ by
\begin{equation*}
S(x_0,\alpha)=\{(x,t)\in\R^{n+1}_+:\|x-x_0\|<\alpha t\}.
\end{equation*}
\begin{defn}\label{nontang}
A function $u$ defined on $\R^{n+1}_+$ or on a strip $\R^n\times(0,t_0)$ for some $t_0>0$, is said to have nontangential limit $L\in\C$ at $x_0\in\R^n$ if, for every $\alpha>0$,
\begin{equation*}
\lim_{\substack{(x,t)\to(x_0,0)\\(x,t)\in S(x_0,\alpha)}}u(x,t)=L.
\end{equation*}
\end{defn}
It is a classical result that if $f\in L^p(\R^n)$, $1\leq p\leq\infty$ then the Poisson integral $Pf$ of $f$ has nontangential limit $f(x_0)$ at each Lebesgue point $x_0$ of $f$ (see \cite[Theorem 3.16]{SW}).
In \cite{Sa}, Saeki generalized this result for more general class of kernels as well as for measures instead of $L^p$-functions (see Theorem \ref{saeki}). A natural question arises that what happens to the nontangential convergence at non-Lebesgue points. To answer this question, Shapiro \cite{Sh} introduced the notion of $\sigma$-point of a locally integrable function.
\begin{defn}\label{sigmap}
A point $x_0\in\R^n$ is called a $\sigma$-point of a locally integrable function $f$ on $\R^n$ provided the following holds: for each $\epsilon>0$, there exists $\delta>0$ such that
\begin{equation*}
\left|\int_{B(x,r)}\left(f(\xi)-f(x_0)\right)\:dm(\xi)\right|<\epsilon(\|x-x_0\|+r)^n,
\end{equation*}
whenever $\|x-x_0\|<\delta$ and $r<\delta$.
\end{defn}
The set of all $\sigma$-points of $f$ is called the $\sigma$-set of $f$.
As observed by Shapiro, the Lebesgue set of a locally integrable function is contained in the $\sigma$-set of the function \cite[P.3182]{Sh}. This containment is strict for some functions. In fact, Shapiro constructed a function $f\in L^p(\R^2)$, $1\leq p\leq\infty$ such that $0$ is a $\sigma$-point of $f$ but not a Lebesgue point of $f$ (see \cite[Section 3]{Sh}). Our main aim in this article is to generalize the following result of Shapiro \cite[Theorem 1]{Sh} for measures.
\begin{theorem}\label{shapiro}
Let $f\in L^p(\R^n)$, $1\leq p\leq\infty$. If $x_0\in\R^n$ is a $\sigma$-point of $f$, then $Pf$ has nontangenial limit $f(x_0)$ at $x_0$.
\end{theorem}
In the next section, we will define the notion of $\sigma$-point of a measure on $\R^n$ and prove a generalization of Theorem \ref{shapiro}. We will consider a wide class of kernels which includes the Poisson kernel and Gauss-Weierstrass kernel and discuss the nontangential behavior of convolutions of these kernels with measures.

It is worth mentioning that Ramey-Ullrich \cite{UR} had also discussed the nontangential behavior of $P\mu$ by considering the strong derivative of a positive measure $\mu$. We will discuss the result of Ramey-Ullrich and the relationship between the strong derivative and $\sigma$-point in the last section.

\section{nontangential convergence of convolution integrals}
Let us start by defining the notion of $\sigma$-point of a measure.
\begin{defn}
Let $\mu$ be a measure on $\R^n$. A point $x_0\in\R^n$ is called a $\sigma$-point of $\mu$ if there exists $L\in\C$ satisfying the following: for each $\epsilon>0$, there exists $\delta>0$ such that
\begin{equation*}
|(\mu-Lm)(B(x,r))|<\epsilon(\|x-x_0\|+r)^n,
\end{equation*}
whenever $\|x-x_0\|<\delta$ and $r<\delta$. In this case, we will denote $D_{\sigma}\mu(x_0)=L$.
\end{defn}
\begin{remark}\label{lebmeanssigma}
Every Lebesgue point of the measure $\mu$ is a $\sigma$-point of $\mu$. Moreover, $D_{\sigma}\mu(x_0)=D_{sym}\mu(x_0)$, whenever $x_0$ is a Lebesgue point of $\mu$. To see this, we take a Lebesgue point $x_0\in\R^n$ of $\mu$. We fix $\epsilon>0$. By the definition of Lebesgue point, there exists $\delta>0$ such that
\begin{equation*}
|\mu-Lm|(B(x_0,r))<\frac{\epsilon}{m(B(0,1))}m(B(x_0,r))=\epsilon r^n,
\end{equation*}
whenever $0<r<\delta$, where $L=D_{sym}\mu(x_0)$. This implies that
\begin{equation*}
|(\mu-Lm)(B(x,r))|\leq |\mu-Lm|(B(x,r))\leq|\mu-Lm|(B(x_0,r+\|x-x_0\|))<\epsilon(\|x-x_0\|+r)^n,
\end{equation*}
whenever $0<r+\|x-x_0\|<\delta$. This shows that $x_0$ is a $\sigma$-point of $\mu$ and
\begin{equation*}
D_{\sigma}\mu(x_0)=D_{sym}\mu(x_0).
\end{equation*}
We have already mentioned in the introduction that the converse is not true.
\end{remark}
The following proposition shows that almost every points of $\R^n$ is a Lebesgue point and hence $\sigma$-point of a measure.
\begin{prop}\label{almostevery}
The Lebesgue set of a measure $\mu$ on $\R^n$ includes almost all (with respect to the Lebesgue measure) points of $\R^n$.
\end{prop}
\begin{proof}
Let $d\mu=f\:dm+d\mu_s$ be the Radon-Nikodym decomposition of $\mu$ with respect to $m$, where $f\in L_{loc}^1(\R^n)$ and $\mu_s\perp m$ (see \cite[P.121-123]{Rureal}). If $L_f$ denotes the Lebesgue set of $f$, then we know that \cite[P.12]{SW} $$m(\R^n\setminus L_f)=0.$$ We observe that $|\mu_s|\perp m$ and hence by \cite[Theorem 7.13]{Rureal},
\begin{equation*}
m(\R^n\setminus  A)=0,\:\:\:\text{where}\:\:\:A=\{x\in\R^n\mid D_{sym}|\mu_s|(x)=0\}.
\end{equation*}
Consequently, $\R^n\setminus (L_f\cap A)$ is of Lebesgue measure zero. Now, for $x_0\in L_f\cap A$ and any $r>0$,
\begin{eqnarray*}
\frac{|\mu-f(x_0)m|(B(x_0,r))}{m(B(0,r))}\leq\frac{1}{m(B(0,r))}\int_{B(x_0,r)}|f(x)-f(x_0)|\:dm(x)+\frac{|\mu_s|(B(x_0,r))}{m(B(0,r))}.
\end{eqnarray*}
Since $x_0\in L_f\cap A$, each summand on the right hand side of the inequality above goes to zero as $r\to 0$. This proves our assertion.
\end{proof}
Let $\phi:\R^n\to [0,\infty)$ be radial and radially decreasing measurable function, that is,
\begin{eqnarray*}
&&\phi (x)=\phi(y),\:\:\:\:\text{if}\:\:\|x\|=\|y\|\:\:\:\\
&&\phi (x)\geq \phi(y),\:\:\:\:\text{if}\:\:\|x\|<\|y\|,
\end{eqnarray*}
with
\begin{equation*}
\int_{\R^n}\phi (x)\:dm(x)=1.
\end{equation*}
For $t>0$, we consider the usual approximate identity
\begin{equation*}
\phi_t(x)=t^{-n}\phi\left(\frac{x}{t}\right),\:\:\:\: x\in\R^n.
\end{equation*}
Given a measure $\mu$, we define the convolution integral $\phi[\mu]$ by
\begin{equation}\label{convint}
\phi[\mu](x,t)=\mu\ast\phi_{t}(x)=\int_{\R^n}\phi_{t}(x-\xi)\:d\mu(\xi),
\end{equation}
whenever the integral exists for $(x,t)\in \R^{n+1}_+$.
\begin{remark}\label{existence}
It was proved in \cite[Remark 1.4]{Sa} that if $\mu$ is a measure on $\R^n$ and $\phi$ is a nonnegative, radially decreasing function on $\R^n$ then the finiteness of $|\mu|\ast \phi_{t_0}(x_0)$ implies the finiteness of $|\mu|\ast\phi_t(x)$ for all $(x,t)\in\R^n\times(0,t_0)$. Note also that if $|\mu|(\R^n)$ is finite then $\mu\ast \phi_{t}(x)$ is well defined for all $(x,t)\in\R^{n+1}_+$.
\end{remark}
In addition to the above, in some of our results we will also assume that $\phi$ is strictly positive and satisfies the following comparison condition \cite[P.134]{Sa}.
\begin{equation}\label{comp}
\sup\:\left\{\frac{\phi_t(x)}{\phi (x)}\mid t\in (0,1), \|x\|>1\right\}<\infty.
\end{equation}
It is easy to see that $P(x,1)$ and the Gaussian
\begin{equation}\label{gaussian}
w(x)=(4\pi)^{-\frac{n}{2}}e^{-\frac{\|x\|^2}{4}},\:x\in\R^n
\end{equation}
satisfy the comparison condition (\ref{comp}) \cite[Example 2.2]{Sar1}.
The following generalization of the nontangential convergence of the Poisson integral $Pf$ was proved by Saeki \cite[Theorem 1.5]{Sa}.
\begin{theorem}\label{saeki}
Suppose $\phi:\R^n\to (0,\infty)$ satisfies the following conditions:
\begin{enumerate}
\item $\phi$ is radial, radially decreasing measurable function with $\|\phi\|_{L^1(\R^n)}=1$.
\item $\phi$ satisfies the condition (\ref{comp}).
\end{enumerate}
Suppose $\mu$ is a measure on $\R^n$ such that $|\mu|\ast \phi_{t_0}(x_1)$ is finite for some $t_0>0$ and $x_1\in\R^n$. Then the convolution integral $\phi[\mu]$ has nontangential limit $D_{sym}\mu(x_0)$ at each Lebesgue point $x_0$ of $\mu$.
\end{theorem}
\begin{remark}
It was shown in \cite[Remark 1.6]{Sa} that the theorem above fails in the absence of condition (\ref{comp}). The assumption that $x_0\in\R^n$ is a Lebesgue point is so strong that one can prove the nontangential convergence of $\psi[\mu]$ (defined analogously to (\ref{convint})) at $x_0$ for any measurable function $\psi:\R^n\rightarrow\C$ (not necessarily radial, radially decreasing) such that $|\psi(x)|\leq\phi(x)$, for all $x\in\R^n$.
\end{remark}
Our main interest in this paper is to prove a generalization of Theorem \ref{shapiro} for convolution integral of the form (\ref{convint}). Our first lemma shows that condition (\ref{comp}) can be used to reduce matters to the case of a measure $\mu$ such that $|\mu|(\R^n)<\infty$.
\begin{lemma}\label{reducenontan}
Suppose $\phi$ is as in Theorem \ref{saeki}. If $\mu$ is a measure such that $|\mu|\ast\phi_{t_0}(0)$ is finite for some $t_0\in (0,\infty)$, then for all $\alpha>0$,
\begin{equation}\label{reducemunontan}
\lim_{\substack{(x,t)\to(0,0)\\(x,t)\in S(0,\alpha)}}\mu\ast\phi_t(x)=\lim_{\substack{(x,t)\to(0,0)\\(x,t)\in S(0,\alpha)}}\tilde{\mu}\ast\phi_t(x),
\end{equation}
where $\tilde{\mu}$ is the restriction of $\mu$ on the closed ball $\overline{B(0,t_0)}$. Moreover, if zero is a $\sigma$-point of $\mu$ then zero is also a $\sigma$-point of $\tilde{\mu}$ and vice versa. In both the cases,
\begin{equation*}
D_{\sigma}\mu(0)=D_{\sigma}\tilde{\mu}(0).
\end{equation*}\end{lemma}
\begin{proof}
In view of Remark \ref{existence}, without loss of generality we assume that $t_0<1$. We write for $0<t<t_0$, $x\in\R^n$,
\begin{equation}\label{rel1nontan}
\mu\ast\phi_t(x)=\tilde{\mu}\ast\phi_t(x)+\int_{\{\xi\in\R^n:\|\xi\|>t_0 \}}\phi_t(x-\xi)\:d\mu (\xi).
\end{equation}
Since $\phi$ is a radial function, we will write for the sake of simplicity $\phi(r)=\phi(\xi)$, whenever $r=\|\xi\|$. For any $r\in (0,\infty)$, we have
\begin{equation*}
\int_{r/2\leq \|\xi\|\leq r}\phi (\xi)\:dm(\xi)\geq \omega_{n-1}\phi (r)\int_{r/2}^rs^{n-1}\:ds=A_nr^n\phi (r),
\end{equation*}
where $\omega_{n-1}$ is the surface area of the unit sphere $S^{n-1}$ and $A_n$ is a positive constant which depends only on the dimension.
Since $\phi$ is an integrable function, the integral on the left hand side converges to zero as $r$ goes to zero and infinity. Hence, it follows that
\begin{equation}\label{rel2}
\lim_{\|\xi\|\to 0}\|\xi\|^n\phi(\xi)=\lim_{\|\xi\|\to \infty}\|\xi\|^n\phi (\xi)=0.
\end{equation}
We denote the integral appearing on the right-hand side of (\ref{rel1nontan}) by $I(x,t)$. We fix $\alpha>0$. We observe that for $0<t<\min\{\frac{1}{2},\:\frac{t_0}{2\alpha}\}$,
\begin{equation*}
\|x-\xi\|\geq\|\xi\|-\|x\|\geq\|\xi\|-\frac{\|\xi\|}{2}=\frac{\|\xi\|}{2},
\end{equation*}
whenever $\|\xi\|>t_0$ and $(x,t)\in S(0,\alpha)$. Therefore, using the fact that $\phi$ is radially decreasing, we obtain for $(x,t)\in S(0,\alpha)\cap\big(\R^n\times(0,\min\{\frac{1}{2},\:\frac{t_0}{2\alpha}\})\big)$,
\begin{eqnarray}
|I(x,tt_0)|&=&(tt_0)^{-n}\left|\int_{\{\xi\in\R^n:\|\xi\|>t_0 \}}\phi\left(\frac{x-\xi}{tt_0}\right)\:d\mu(\xi)\right|\nonumber\\
&\leq&(tt_0)^{-n}\int_{\{\xi\in\R^n:\|\xi\|>t_0 \}}\phi\left(\frac{x-\xi}{tt_0}\right)\:d|\mu|(\xi)\nonumber\\
&\leq&(tt_0)^{-n}\int_{\{\xi\in\R^n:\|\xi\|>t_0 \}}\phi\left(\frac{\xi}{2tt_0}\right)\:d|\mu|(\xi)\nonumber\\
&=&\int_{\{\xi\in\R^n:\|\xi\|>t_0 \}}\frac{\left(\frac{\|\xi\|}{tt_0}\right)^n\phi\left(\frac{\xi}{2tt_0} \right)}{\|\xi\|^n\phi_{t_0}(\xi)}\phi_{t_0}(\xi)\:d|\mu|(\xi)\label{rel3nontan}
\end{eqnarray}
From (\ref{rel2}) we get that
\begin{equation*}
\lim_{t\to 0}\left(\frac{\|\xi\|}{tt_0}\right)^n\phi\left(\frac{\xi}{2tt_0}\right)=0,
\end{equation*}
for each fixed $\xi\in\R^n$. On the other hand, by the comparison condition (\ref{comp}), there exists some positive constant $C$ such that
\begin{equation*}
\frac{\left(\frac{\|\xi\|}{tt_0}\right)^n\phi\left(\frac{\xi}{2tt_0} \right)}{\|\xi\|^n\phi_{t_0}(\xi)}=2^n\frac{\phi_{2t}\left(\frac{\xi}{t_0}\right)}{\phi\left(\frac{\xi}{t_0}\right)}\leq C,
\end{equation*}
for $\|\xi\|>t_0$ and $0<t<1/2$. Since $|\mu|\ast\phi_{t_0}(0)<\infty$, that is, $\phi_{t_0}\in L^1(\R^n,d|\mu|)$, by the dominated convergence theorem, it follows from (\ref{rel3nontan}) that
\begin{equation*}
\lim_{\substack{(x,t)\to(0,0)\\(x,t)\in S(0,\alpha)}}|I(x,tt_0)|=0.
\end{equation*}
Consequently,
\begin{equation*}
\lim_{\substack{(x,t)\to(0,0)\\(x,t)\in S(0,\alpha)}}\int_{\{\xi\in\R^n:\|\xi\|>t_0 \}}\phi_t(x-\xi)\:d\mu (\xi)=\lim_{\substack{(x,t)\to(0,0)\\(x,t)\in S(0,\alpha)}}I(x,t)=\lim_{\substack{(x,t)\to(0,0)\\(x,t)\in S(0,\alpha)}}I(x,tt_0^{-1}t_0)=0,
\end{equation*}
as $0<t_0<1$. This proves (\ref{reducemunontan}). Suppose that $D_{\sigma}\mu(0)=L$. We take $\epsilon>0$. Then there exists $0<\delta<t_0/2$ such that
\begin{equation*}
|(\mu-Lm)(B(x,r))|<\epsilon(\|x\|+r)^n,
\end{equation*}
whenever $\|x\|<\delta$ and $r<\delta$. But for $\|x\|<\delta$ and $r<\delta$, we observe that
\begin{equation*}
B(x,r)\subset B(0,2\delta)\subset B(0,t_0).
\end{equation*}
Using this observation and the definition of $\tilde{\mu}$ in the last inequality, we get that
\begin{equation*}
|(\tilde{\mu}-Lm)(B(x,r))|<\epsilon(\|x\|+r)^n,
\end{equation*}
whenever $\|x\|<\delta$ and $r<\delta$. This shows that
\begin{equation*}
D_{\sigma}\tilde{\mu}(0)=L=D_{\sigma}\mu(0).
\end{equation*}
Proof of the converse implication is similar.
\end{proof}
Before proceed to our next lemma, we recall that a real valued function $f$ on a topological space $X$ is said to be lower semicontinuous if $\{x\in X:f(x)>s\}$ is open for every real number $s$ \cite[P.37]{Rureal}.
\begin{lemma}\label{ball}
Assume that $\phi:\R^n\rightarrow[0,\infty)$ is a radial, radially decreasing, integrable function. If $\phi$ is lower semicontinuous then, for every $t\in(0,\phi(0))$
\begin{equation*}
B_t=\{x\in\R^n:\phi(x)>t\},
\end{equation*}
is an open ball centred at zero with some finite radius $\theta(t)$ (say).  
\end{lemma}
\begin{proof}
Since $\phi$ is integrable, there exists $x_0\in\R^n$ such that $\phi(x_0)\leq t$. For any $x\in B_t$,
\begin{equation*}
\phi(x)>t\geq\phi(x_0).
\end{equation*}
As $\phi$ is radially decreasing, the inequality above implies that $\|x\|<\|x_0\|$ and hence $B_t$ is bounded. Therefore,
\begin{equation*}
\theta(t):=\sup\{r>0:\overline{B(0,r)}\subset B_t\}<\infty.
\end{equation*}
We claim that $B(0,\theta(t))$ is contained in $B_t$. To see this, for $x\in B(0,\theta(t))$, we take $r\in(\|x\|,\theta(t))$.  By the definition of $\theta(t)$, this implies that $x\in B(0,r)\subset B_t$. On the other hand, if $x\in B_t$,
\begin{equation*}
\phi(\xi)\geq\phi(x)>t,\:\:\:\text{for all}\:\:\xi\in\overline{B(0,\|x\|)}.
\end{equation*}
Thus, $\overline{B(0,\|x\|)}\subset B_t$ for all $x\in B_t$. Consequently, 
\begin{equation}\label{setcontainment}
B(0,\theta(t))\subset B_t\subset\overline{B(0,\theta(t))}.
\end{equation} We shall show that  
\begin{equation*}
B_t=B(0,\theta(t))\:\:\:\text{or}\:\:\:B_t=\overline{B(0,\theta(t))}.
\end{equation*}
Suppose there exists $\xi\in B_t\setminus B(0,\theta(t))$. Then by (\ref{setcontainment}), $\|\xi\|=\theta(t)$ and hence by radiality of $\phi$, $B_t=\overline{B(0,\theta(t))}$. Since $\phi$ is lower semicontinuous, $B_t=B(0,\theta(t))$.
\end{proof}
\begin{remark}
It follows from the proof that if $\phi$ is not a lower semicontinuous function then $B_t$ may turn out to be a closed ball centred at origin. This can be seen from the following example. Define $\phi:\R^n\rightarrow(0,\infty)$ by
\begin{equation*}
\phi(x)=\begin{cases}
e^{-\|x\|},&\|x\|\leq 1\\e^{-2\|x\|},&\|x\|>1.
\end{cases}
\end{equation*}
Then for any $t\in(e^{-2},e^{-1})$, $B_t=\overline{B(0,1)}$.
\end{remark}
We are now ready to present our main result which generalizes Theorem \ref{shapiro}.
\begin{theorem}\label{mythmsigma}
Suppose $\phi$ and $\mu$ be as in Theorem \ref{saeki}. Further assume that $\phi$ is lower semicontinuous. If $x_0\in\R^n$ is a $\sigma$-point of $\mu$ with $D_{\sigma}\mu(x_0)=L\in\C$, then $\phi[\mu]$ has nontangential limit $L$ at $x_0$.
\end{theorem}
\begin{proof}
Without loss of generality, we can assume $x_0=0$. Indeed, we consider the translated measure $\mu_0=\tau_{-x_0}\mu$, where
\begin{equation*}
\tau_{-x_0}\mu (E)=\mu(E+x_0),
\end{equation*}
for all Borel subsets $E\subset \R^n$. Using translation invariance of the Lebesgue measure it follows that
\begin{equation*}
(\mu_0-Lm)(B(x,r))=(\mu-Lm)(B(x+x_0,r)).
\end{equation*}
We fix $\epsilon>0$. Since $x_0$ is a $\sigma$-point of $\mu$ with $D_{\sigma}\mu(x_0)=L$, the equality above implies that there exists $\delta>0$ such that
\begin{equation*}
|(\mu_0-Lm)(B(x,r))|<\epsilon(\|x\|+r)^n,\:\:\:\text{whenever}\:\:\|x\|<\delta,\:\: r<\delta.
\end{equation*}
This shows that $0$ is a $\sigma$-point of $\mu_0$ with $D_{\sigma}\mu_0(0)=L$. As translation commutes with convolution, it also follows that
\begin{equation}\label{transphi}
\mu_0\ast\phi_t(x)=(\tau_{-x_0}\mu\ast\phi_t)(x)=\tau_{-x_0}(\mu\ast\phi_t )(x)=\mu\ast\phi_t(x+x_0),
\end{equation}
for any $(x,t)\in\R^n\times (0,t_0)$. We fix an arbitrary positive number $\alpha$. As $(x,t)\in S(0,\alpha)$ if and only if $(x_0+x,t)\in S(x_0,\alpha)$, one infers from (\ref{transphi}) that
\begin{equation*}
\lim_{\substack{(x,t)\to(0,0)\\(x,t)\in S(0,\alpha)}}\phi[\mu_0](x,t)=\lim_{\substack{(\xi,t)\to(x_0,0)\\(\xi,t)\in S(x_0,\alpha)}}\phi[\mu](\xi,t).
\end{equation*}
Hence, it suffices to prove the theorem under the assumption that $x_0=0$. Applying Lemma \ref{reducenontan}, we can restrict $\mu$ on $\overline{B(0,t_0)}$, if necessary, to assume that $|\mu|(\R^n)<\infty$. Since $D_{\sigma}\mu(0)=L$,
\begin{equation*}
\lim_{r\to 0}\frac{\mu(B(0,r))}{m(B(0,r))}=L.
\end{equation*}
Therefore, there exists a positive constant $r_0$ such that
\begin{equation*}
\frac{|\mu(B(0,r))|}{m(B(0,r))}<L+1,\:\:\:\text{for all}\:\:r<r_0.
\end{equation*}
Using finiteness of the total variation of $\mu$, we get that
\begin{equation*}
\frac{|\mu(B(0,r))|}{m(B(0,r))}\leq\frac{|\mu|(B(0,r))}{m(B(0,r))}\leq\frac{|\mu|(\R^n)}{m(B(0,r_0))},\:\:\:\text{for all}\:\:r\geq r_0.
\end{equation*}
Combining above two inequalities, we obtain
\begin{equation}\label{boundofmu}
M{\mu}(0):=\sup_{r>0}\frac{|\mu(B(0,r))|}{m(B(0,r))}<\infty.
\end{equation}
For each $0<t<\phi(0)$, we define
\begin{equation*}
B_t=\{x\in\R^n:\phi(x)>t\}.
\end{equation*}
By Lemma \ref{ball}, $B_t$ is an open ball with centre at $0$ and radius $\theta(t)$. Clearly, $\theta$ is a monotonically decreasing function in $(0,\phi(0))$ and hence measurable. We also note that for any $r\in(0,\infty)$ and $x\in\R^n$,
\begin{equation*}
\Big\{\xi\in\R^n:\phi\left(\frac{x-\xi}{r}\right)>t\Big\}
\end{equation*}
is an open ball with centre at $x$ and radius $r\theta(t)$. Let $\{(x_k,t_k)\}_{k=1}^{\infty}$ be a sequence in $S(0,\alpha)$ converging to $(0,0)$. Without loss of generality we assume that $t_k\in (0,t_0)$ for all $k$. As $\int_{\R^n}\phi(x)\:dm(x)=1$, we can write
\begin{eqnarray}
\mu\ast\phi_{t_k}(x_k)-L&=&t_k^{-n}\int_{\R^n}\phi\left(\frac{x_k-\xi}{t_{k}}\right)\:d\mu(\xi)-Lt_k^{-n}\int_{\R^n}\phi\left(\frac{x_k-\xi}{t_{k}}\right)\:dm(\xi)\nonumber\\
&=&t_k^{-n}\int_{\R^n}\phi\left(\frac{x_k-\xi}{t_{k}}\right)\:d(\mu-Lm)(\xi)\nonumber\\
&=&t_k^{-n}\int_{\R^n}\int_{0}^{\phi\left(\frac{x_k-\xi}{t_{k}}\right)}\:ds\:d(\mu-Lm)(\xi).\nonumber
\end{eqnarray}
As $|\mu-Lm|\ast\phi_t(x)$ is finite for all $(x,t)\in\R^n\times(0,t_0)$, applying Fubini's theorem on the right hand side of the last equality, we obtain
\begin{eqnarray}
\mu\ast\phi_{t_k}(x_k)-L&=&t_k^{-n}\int_{0}^{\phi(0)}\left(\mu-Lm\right)\left(\Big\{\xi\in\R^n:\phi\left(\frac{x_k-\xi}{t_k}\right)>s\Big\}\right)\:ds\nonumber\\
&=&\int_{0}^{\phi(0)}\frac{(\mu-Lm)\left(B(x_k,t_k\theta(s))\right)}{\left(\|x_k\|+t_k\theta(s)\right)^n}\times\left(\frac{\|x_k\|+t_k\theta(s)}{t_k}\right)^n\:ds\label{final}.
\end{eqnarray}
Since $D_{\sigma}\mu(0)=L$,
\begin{equation*}
\lim_{(x,r)\to(0,0)}\frac{(\mu-Lm)(B(x,r))}{(\|x\|+r)^n}=0.
\end{equation*}
Therefore, for each $s\in (0,\phi(0))$, integrand on the right hand side of (\ref{final}) has limit zero as $k\to\infty$ because $\|x_k\|/t_k<\alpha$, for all $k$. Moreover, using (\ref{boundofmu}), the integrand is bounded by the function
\begin{equation*}
s\mapsto m(B(0,1))(M\mu(0)+L)(\theta(s)+\alpha)^n,\:\:\:\: s\in(0,\phi(0)).
\end{equation*}
In order to apply the dominated convergence theorem on the right hand side of (\ref{final}), we need to show that this function is integrable in $(0,\phi(0))$. For this, it is enough to show that the function $s\mapsto\theta(s)^n$ is integrable in $(0,\phi(0))$. Using a well-known formula involving distribution functions \cite[Theorem 8.16]{Rureal}, we observe that
\begin{eqnarray*}
\int_{\R^n}\phi(x)\:dm(x)&=&\int_{0}^{\phi(0)}m\left(\{x\in\R^n:\phi(x)>s\}\right)\:ds\\
&=&\int_{0}^{\phi(0)}m(B_s)\:ds\\
&=&m(B(0,1))\int_{0}^{\phi(0)}\theta(s)^n\:ds.
\end{eqnarray*}
Hence, applying the dominated convergence theorem on the right hand side of (\ref{final}) we obtain
\begin{equation*}
\lim_{k\to\infty}\mu\ast\phi_{t_k}(x_k)=L.
\end{equation*}
This completes the proof.
\end{proof}
Shapiro also considered nontangential limits of Gauss-Weierstrass integral of a $L^p$-function \cite[Theorem 2]{Sh}. We recall that the Gauss-Weierstrass kernel or the heat kernel of $\R^{n+1}_+$ is given by
\begin{equation*}
W(x,t)=(4\pi t)^{-\frac{n}{2}}e^{-\frac{\|x\|^2}{4t}},\:(x,t)\in\R^{n+1}_+.
\end{equation*}
The Gauss-Weierstrass integral of a measure $\mu$ is given by the convolution
\begin{equation*}
W\mu(x,t)=\int_{\R^n}W(x-y,t)\:d\mu(y),\:\:\:\: x\in\R^n,\:\: t\in (0,\infty ),
\end{equation*}
whenever the above integral exists. Recalling (\ref{gaussian}), we observe that
\begin{equation}\label{gaussint}
W\mu(x,t)=\mu\ast w_{\sqrt{t}}(x),\:\:\:\:(x,t)\in\R^{n+1}_+.
\end{equation}
As an easy corollary of Theorem \ref{mythmsigma}, we get the following generalization of the above mentioned theorem of Shapiro.
\begin{corollary}
Suppose $\mu$ is a measure on $\R^n$ such that $W|\mu|(x_1,t_0)$ is finite for some $x_1\in\R^n$ and $t_0>0$. If $x_0\in\R^n$ is a $\sigma$-point of $\mu$ with $D_{\sigma}\mu(x_0)=L\in\C$, then the Gauss-Weierstrass integral $W\mu$ has nontangential limit $L$ at $x_0$.
\end{corollary}
\begin{proof}
We fix an arbitrary positive number $\alpha$. We have already mentioned that $w$ satifies the comparison condition (\ref{comp}). Moreover, $\|w\|_{L^1(\R^n)}=1$ (see \cite[P.9]{SW}). Thus, $w$ satisfies all the hypothesis of Theorem \ref{mythmsigma}. Hence, in view of (\ref{gaussint}), Theorem \ref{mythmsigma} gives
\begin{equation*}
\lim_{\substack{(x,t)\to(x_0,0)\\\|x-x_0\|<\sqrt{\alpha t}}}W\mu(x,t)=L.
\end{equation*}
Note that
\begin{equation*}
S(x_0,\alpha)\cap\{(x,t)\in\R^{n+1}_+\mid t<\frac{1}{\alpha}\}\subset\{(x,t)\in\R^{n+1}_+\mid \|x-x_0\|<\sqrt{\alpha t},\:\:t<\frac{1}{\alpha}\}.
\end{equation*}
Using this set containment relation together with the equation above, we conclude that $W\mu$ has nontangential limit $L$ at $x_0$.
\end{proof}
We can drop the comparison condition (\ref{comp}) in Theorem \ref{mythmsigma} by imposing some growth condition on $\mu$. More precisely, we have the following.
\begin{theorem}\label{mythmgrowth}
Let $\phi:\R^n\to [0,\infty)$ be radial, radially decreasing, lower semicontinuous function with $\|\phi\|_{L^1(\R^n)}=1$. Suppose $\mu$ is a measure on $\R^n$ such that
\begin{equation}\label{growth}
|\mu|(B(0,r))=O(r^n),\:\: as\:\:r\to\infty,
\end{equation}
and that $\mu\ast\phi_{t_0}(x_1)$ is finite for some $x_1\in\R^n$ and $t_0\in(0,\infty)$. If $x_0\in\R^n$ is a $\sigma$-point of $\mu$ with $D_{\sigma}\mu(x_0)=L\in\C$, then $\phi[\mu]$ has nontangential limit $L$ at $x_0$.
\end{theorem}
\begin{proof}
Without loss of generality, we assume that $x_0=0$. We will use the same notation as in the proof of Theorem \ref{mythmsigma}. From the proof of Theorem \ref{mythmsigma}, we observe that it suffices to prove that $M\mu(0)<\infty$ and then the
the rest of the arguments remains same. As $D_{\sigma}\mu(0)=L$, it follows that $D_{sym}\mu(0)=L$ and hence there exists a positive constant $r_0$ such that
\begin{equation*}
\frac{|\mu(B(0,r))|}{m(B(0,r))}<L+1,\:\:\:\text{for all}\:\:r\leq r_0.
\end{equation*}
Using (\ref{growth}), we get two positive constants $M_0$ and $R_0$ such that
\begin{equation*}
\frac{|\mu(B(0,r))|}{m(B(0,r))}<M_0,\:\:\:\text{for all}\:\:r\geq R_0.
\end{equation*}
Finally, for all $r\in(r_0,R_0)$
\begin{equation*}
\frac{|\mu(B(0,r))|}{m(B(0,r))}\leq\frac{|\mu|(B(0,R_0))}{m(B(0,r_0))}
\end{equation*}
From the last three inequalities and the fact that $|\mu|$ is locally finite, we conclude that
\begin{equation*}
M\mu(0)=\sup_{r>0}\frac{|\mu(B(x_0,r))|}{m(B(0,r))}<\infty.
\end{equation*}
\end{proof}
\begin{remark}
We can drop the assumption that $\phi$ is lower semicontinuous from Theorem \ref{mythmsigma} and Theorem \ref{mythmgrowth} in the following two special cases.
\begin{enumerate}
\item [i)] $x_0$ is a Lebesgue point of $\mu$.
\item [ii)] $\mu$ is absolutely continuous with respect to the Lebesgue measure $m$. 
\end{enumerate}
\end{remark}
\section{$\sigma$-point and strong derivative}
In this section, we will discuss the relationship between $\sigma$-point of a measure and the notion of strong derivative of a measure introduced by Ramey-Ullrich \cite{UR}. We recall the definition of strong derivative of a measure.
\begin{defn}
Given a measure $\mu$ on $\R^n$, we say that $\mu$ has strong derivative $L\in\C$ at $x_0\in\R^n$ if
\begin{equation*}
\lim_{r\to 0}\frac{\mu(x_0+rB)}{m(rB)}=L
\end{equation*}
holds for every open ball $B\subset\R^n$. Here, $rB=\{rx\mid x\in B\}$, $r>0$. The strong derivative of $\mu$ at $x_0$, if it exists, is denoted by $D\mu(x_0)$. Note that $rB(\xi,s)=B(r\xi,rs)$.
\end{defn}
\begin{prop}\label{sigmaimstrong}
Let $\mu$ be a measure on $\R^n$. If $x_0\in\R^n$ is a $\sigma$-point of $\mu$ with $D_{\sigma}\mu(x_0)=L\in\C$, then the strong derivative of $\mu$ at $x_0$ exists and is equal to $L$.
\end{prop}
\begin{proof}
We take a ball $B=B(x,s)$ in $\R^n$ and fix $\epsilon>0$. As $x_0$ is a $\sigma$-point of $\mu$, there exists $\delta>0$ such that
\begin{equation*}
|(\mu-Lm)(x_0+rB)|=|(\mu-Lm)\left(B(x_0+rx,rs)\right)|<\epsilon(\|rx\|+rs)^n,
\end{equation*}
whenever $\|rx\|<\delta$ and $rs<\delta$. This implies that
\begin{equation*}
\left|\frac{\mu(x_0+rB)}{m(rB)}-L\right|<\epsilon\frac{(\|x\|+s)^n}{m(B(0,s))},\:\:\:\text{whenever}\:\: \|rx\|<\delta\:\:\text{and}\:\: rs<\delta.
\end{equation*}
Taking $r_0=\min\{\frac{\delta}{\|x\|+1},\:\frac{\delta}{s}\}$, it follows that the last inequality holds for all $r<r_0$. This completes the proof.
\end{proof}
\begin{remark}\label{Ramey-Ullrich}
In \cite[Theorem 2.2]{UR}, among other things, Ramey-Ullrich proved that if $\mu$ is a positive measure on $\R^n$ with well-defined Poisson integral $P\mu$ then the strong derivative of $\mu$ at $x_0\in\R^n$ is $L\in[0,\infty)$ if and only if $P\mu$ have nontangential limit $L$ at $x_0$. In view of Proposition \ref{sigmaimstrong}, we can deduce Theorem \ref{mythmsigma} for $\phi=P(.,1)$ and $\mu$ positive from the result of Ramey-Ullrich.
\end{remark}
The converse of Proposition \ref{sigmaimstrong} is true in one dimension. If $\mu$ is a locally finite signed measure on $\R$ then there is a function $f:\R\to\R$ of bounded variation such that the positive and negative parts of $f$ are right continuous and
\begin{equation*}
\mu\left((a,b]\right)=f(b)-f(a),\:\:\:a,\:b\in\R,\:a<b.
\end{equation*}
For more discussion on this see \cite[P.281-284]{SS}.
\begin{prop}\label{onedim}
Suppose that $\mu$ and $f$ as above and $x_0\in\R$.
\begin{enumerate}
\item[i)] The function $f$ is differentiable at $x_0$ if and only if $x_0$ is a $\sigma$-point of $\mu$. In this case, $f^{\prime}(x_0)=D_{\sigma}\mu(x_0)$.
\item[ii)] The function $f$ is differentiable at $x_0$ if and only if the strong derivative $\mu$ at $x_0$ exists. In this case, $f^{\prime}(x_0)=D\mu(x_0)$.
\end{enumerate}
\end{prop}
\begin{proof}
We first prove $i)$. Suppose $f$ is differentiable at $x_0$ and $f^{\prime}(x_0)=L\in\R$. Fix $\epsilon>0$ and choose $\delta>0$ such that
\begin{equation}\label{fdifferentiable}
\left|\frac{f(x_0+h)-f(x_0)}{h}-L\right|<\epsilon,\:\:\:\text{whenever}\:\:|h|<\delta.
\end{equation}
For $x\in\R$, $r>0$ with $|(x-x_0)+r|<\delta$ and $|(x-x_0)-r|<\delta$, we have
\begin{eqnarray*}
&&\left|\left(\mu-Lm\right)((x-r,x+r))\right|\\&=&|f(x+r)-f(x-r)-2rL|\\
&=&\left|\frac{f(x_0+x-x_0+r)-f(x_0)}{x-x_0+r}\times(x-x_0+r)-(x-x_0+r)L+(x-x_0-r)L\right.\\
&&\:\:\:\:\:\left.-\frac{f(x_0+x-x_0-r)-f(x_0)}{x-x_0-r}\times(x-x_0-r)\right|\\
&\leq&|x-x_0+r|\left|\frac{f(x_0+x-x_0+r)-f(x_0)}{x-x_0+r}-L\right|\\
&&\:\:+|x-x_0-r|\left|\frac{f(x_0+x-x_0-r)-f(x_0)}{x-x_0-r}-L\right|\\
&<&|x-x_0+r|\epsilon+|x-x_0-r|\epsilon\:\:\:\:(\text{by}\:(\ref{fdifferentiable})).
\end{eqnarray*}
This implies that
\begin{equation*}
|(\mu-Lm)(B(x,r))|<2\epsilon(|x-x_0|+r),\:\:\:\text{whenever}\:\:|x-x_0|<\delta/2,\: r<\delta/2.
\end{equation*}
Thus, $x_0$ is a $\sigma$-point of $\mu$ with $D_{\sigma}\mu(x_0)=L$.

Conversely, we assume that $x_0$ is a $\sigma$-point of $\mu$ with $D_{\sigma}\mu(x_0)=L\in\R$ and fix $\epsilon>0$. Then there exists $\delta>0$ such that
\begin{equation}\label{sigmadefn}
|(\mu-Lm)((x-r,x+r))|<\epsilon(|x-x_0|+r),\:\:\:\text{whenever}\:\:|x-x_0|<\delta,\: r<\delta.
\end{equation}
Taking $x=x_0+r$ with $r>0$ in (\ref{sigmadefn}), we obtain
\begin{eqnarray*}
|\mu((x_0,x_0+2r))-2rL|&=& |f(x_0+2r)-f(x_0)-2rL|\\
&=& 2r\left|\frac{f(x_0+2r)-f(x_0)}{2r}-L\right|<\epsilon,
\end{eqnarray*}
whenever $r<\delta$. This shows that $f^{\prime}(x_0+)=L$. Similarly, by taking $x=x_0-r$ with $r>0$ in (\ref{sigmadefn}), we get that $f^{\prime}(x_0-)=L$.

The statement $ii)$ can be proved by arguing in a similar fashion. We refer the reader to \cite[Remark 2.6 (2)]{Sar2} where it was proved under the assumption that $f$ is monotonically increasing.
\end{proof}
Considering real and imaginary parts of a measure, if necessary, we obtain the following corollary.
\begin{corollary}\label{sigmaequalstrong}
Suppose $\mu$ is a measure on $\R$ and $x_0\in\R$. Then $x_0$ is a $\sigma$-point of $\mu$ if and only if $\mu$ has strong derivative at $x_0$. Moreover, $D_{\sigma}\mu(x_0)=D\mu(x_0)$.
\end{corollary}
\begin{remark}
\begin{enumerate}
\item [i)]It is not known to us whether for a measure $\mu$ on $\R^n$, the $\sigma$-set of $\mu$ coincides with the set of points at which the strong derivative of $\mu$ exists, if $n>1$. It would be surprising if it is true in higher dimensions. A heuristic reasoning behind this is the following observation. Suppose $\mu$ is a measure on $\R^n$. If $0$ is a $\sigma$-point of $\mu$ then
\begin{equation*}
(\mu-D_{\sigma}(0)m)(B(x,r))\to 0,\:\:\:\:\text{as}\;\:(x,r)\to (0,0). 
\end{equation*}
On the other hand, existence of strong derivative at $0$ only ensures 
\begin{equation*}
(\mu-D\mu(0)m)(B(x,r))\to 0,\:\:\:\:\text{as}\:\: (x,r)\to (0,0), 
\end{equation*}
along the rays of the form $\{(rx_0,rt_0)\mid r>0\}$, where $(x_0,t_0)\in\R^{n+1}_+$.
\item [ii)] Suppose $\mu$ and $\phi$ as in Theorem \ref{mythmsigma} and $n=1$. If $D\mu(x_0)=L$ then it follows from Proposition \ref{onedim} and Theorem \ref{mythmsigma} that $\phi[\mu]$ converges nontangentially to $L$. It is not known whether the same is true for dimension $n>1$. However, the following theorem shows that a weaker version of convergence for $\phi[\mu]$ holds at the points where the strong derivative $D\mu$ exists.
\end{enumerate}
\end{remark}
\begin{theorem}
Let $\phi$ and $\mu$ be as in Theorem \ref{mythmsigma}. Suppose $\mu$ has strong derivative $L\in\C$ at $x_0\in\R^n$. Then $\phi[\mu](x,t)$ has limit $L$ as $(x,t)\to (x_0,0)$ along each ray through $(x_0,0)$ in $\R^{n+1}_+$. In other words,
\begin{equation*}
\lim_{r\to 0}\phi[\mu](x_0+r\xi,r\eta)=L,\:\:\:\text{for each fixed}\:\:(\xi,\eta)\in\R^{n+1}_+.
\end{equation*}
\end{theorem}
\begin{proof}
Without loss of generality, we can assume $x_0=0$. Let $\tilde{\mu}$ be the restriction of $\mu$ on the ball $\overline{B(0,t_0)}$. If $B(y,\tau)$ is any given ball, then for all $0<r<t_0(\tau+\|y\|)^{-1}$, it follows that $rB(y,\tau)$ is contained in $B(0,t_0)$. This in turn implies that $D\mu(0)$ and $D\tilde{\mu}(0)$ are equal. Thus, in view of Lemma \ref{reducenontan}, without loss of generality, we can assume that $|\mu|(\R^n)$ is finite. We will use the same notation as in the proof of Theorem \ref{mythmsigma}. Since $D\mu(0)$ is equal to $L$, it follows that $D_{sym}\mu(0)$ is also equal to $L$ and hence $M\mu(0)$ is finite (see the argument preceding (\ref{boundofmu})). We take $(\xi,\eta)\in\R^{n+1}_+$ and a sequence $\{r_k\}$ of positive numbers converging to zero. Substituting $x_k=r_k\xi$, $t_k=r_k\eta$ in equation (\ref{final}), we obtain
\begin{equation}\label{final1}
\phi[\mu](r_k\xi,r_k\eta)-L=\int_{0}^{\phi(0)}\frac{(\mu-Lm)\left(B(r_k\xi,r_k\eta\theta(s))\right)}{\left(r_k\eta\theta(s)\right)^n}\theta(s)^n\:ds.
\end{equation}
As $D\mu(0)=L$, using the definition of strong derivative, we observe that for each fixed $s\in (0,\phi(0))$ 
\begin{equation}\label{limittoapplydct}
\lim_{k\to\infty}\frac{(\mu-Lm)(B(r_k\xi,r_k\eta\theta(s)))}{(r_k\eta\theta(s))^n}=\lim_{k\to\infty}\left(\frac{\mu(r_kB(\xi,\eta\theta(s)))}{m(r_kB(\xi,\eta\theta(s)))}-L\right)c_n^{\prime}=0,
\end{equation}
where $c_n^{\prime}=m(B(0,1))$. The integrand on the right hand side of (\ref{final1}) is bounded by the function
\begin{equation*}
s\mapsto m(B(0,1))(M\mu(0)+L)\theta(s)^n,\:\:\:\: s\in(0,\phi(0)).
\end{equation*}
We have seen in the proof of Theorem \ref{mythmsigma} that this function is integrable in $(0,\phi(0))$. In view of (\ref{limittoapplydct}), we can now apply dominated convergence theorem on the right-hand side of (\ref{final1}) to complete the proof.
\end{proof}
We show by an example that the existence of limit of $\phi[\mu]$ along every ray through $(x_0,0)$ may not imply the existence of the strong derivative at $x_0$.
\begin{example}
Consider the measure $d\mu=\chi_{[0,1]}dm$ on $\R$. Then $D_{sym}\mu(0)$ is $1/2$ but the strong derivative $D\mu$ does not exist at the origin (see \cite[Remark 2.5]{Sar2}). Taking $\phi=P(.,1)$, we see that
\begin{equation*}
\phi[\mu](x,t)=\frac{1}{\pi}\int_{0}^{1}\frac{t}{t^2+(x-\xi)^2}\:dm(\xi)=\frac{1}{\pi}\left(\arctan\frac{1-x}{t}+\arctan\frac{x}{t}\right),\:\:\:(x,t)\in\R^{n+1}_+.
\end{equation*}
Therefore, for each fixed $(\xi_0,t_0)\in\R^{n+1}_+$ we have
\begin{equation*}
\lim_{r\to 0}\phi[\mu](r\xi_0,rt_0)=\lim_{r\to 0}\frac{1}{\pi}\left(\arctan\frac{1-r\xi_0}{rt_0}+\arctan\frac{r\xi_0}{rt_0}\right)=\frac{1}{\pi}\left(\frac{\pi}{2}+\arctan\frac{\xi_0}{t_0}\right).
\end{equation*}
This shows that $\phi[\mu]$ has limit along every ray through the origin but the limit depends on the ray.
\end{example}
\section*{acknowledgements}
The author would like to thank Swagato K. Ray for many
useful discussions during the course of this work. The author is supported by a research fellowship from Indian Statistical Institute.

\end{document}